\documentclass{amsart}

\usepackage{amsmath,amssymb,amsfonts,enumerate,amsthm,graphicx,color,hyperref}
\usepackage{tikz, float} \usetikzlibrary {positioning}
\renewcommand{\dim}{\mbox{dim}\,}
\renewcommand{\dim}{\mbox{dim}\,}

\newcommand{\set}{\mbox{set}\,}

\newcommand{\bight}{\mbox{bight}\,}

\newcommand{\T}{\mathrm}

\newtheorem{thm}{Theorem}[section]
\newtheorem{cor}[thm]{Corollary}
\newtheorem{lem}[thm]{Lemma}
\newtheorem{prop}[thm]{Proposition}
\newtheorem{defn}[thm]{Definition}

\newtheorem{exam}[thm]{Example}
\newtheorem{rem}[thm]{Remark}

\numberwithin{equation}{section}
\tikzstyle{Cgray}=[draw=black, scale = .4,circle, fill = white, minimum size=7mm]
\tikzstyle{Cwhite}=[scale = .8,circle, fill = white, minimum size=8mm]
\tikzstyle{Cblack}=[scale = .3,circle, fill = black, minimum size=3mm]
\tikzstyle{C0}=[scale = .9,circle, fill = black!0, inner sep = 0pt, minimum size=3mm]
\tikzstyle{C1}=[scale = .7,circle, fill = black!0, inner sep = 0pt, minimum size=3mm]
\tikzstyle{CW}=[scale = .7,circle, fill = white!0, inner sep = 0pt, minimum size=7mm]
\tikzstyle{Cg}=[scale = .8,circle, fill = gray, minimum size=8mm]

\begin{document}
\bibliographystyle{amsplain}

%\date{}
\title[Expansion of a simplicial complex]{Expansion of a simplicial complex}
\author[S. Moradi and F. Khosh-Ahang]{Somayeh Moradi and Fahimeh Khosh-Ahang}
%%\thanks{}
%\subjclass{13D02, 13P10, 13D40, 13A02}
\address{Somayeh Moradi, Department of Mathematics,
 Ilam University, P.O.Box 69315-516, Ilam, Iran and School of Mathematics, Institute
 for Research in Fundamental Sciences (IPM), P.O.Box: 19395-5746, Tehran, Iran.} \email{somayeh.moradi1@gmail.com}
\address{Fahimeh Khosh-Ahang, Department of Mathematics,
 Ilam University, P.O.Box 69315-516, Ilam, Iran.}
\email{fahime$_{-}$khosh@yahoo.com}

\keywords{Cohen-Macaulay, edge ideal, expansion, projective dimension, regularity, shellable, vertex
decomposable.\\
%$*$Email: somayeh.moradi1@gmail.com/ telephone: +988422227022
}

\begin{abstract}
\noindent
For a simplicial complex $\Delta$, we introduce a simplicial complex attached to $\Delta$, called the expansion of $\Delta$, which is a natural generalization of the notion of expansion in graph theory. We are interested in knowing how the  properties of  a
simplicial complex and  its Stanley-Reisner ring  relate to those of its expansions. It is shown that taking expansion preserves vertex decomposable and
shellable properties and in some cases Cohen-Macaulayness. Also it is proved that some homological invariants of Stanley-Reisner ring of a simplicial complex relate to those invariants in the
Stanley-Reisner ring of its expansions.
\end{abstract}

\maketitle

\section*{Introduction}

Simplicial complexes are widely used structures which have many applications in algebraic topology and commutative algebra. In particular, in order to characterize monomial quotient rings with a desired property, simplicial complex is a very strong tool considering the Stanley-Reisner correspondence between simplicial complexes and monomial ideals.
  Characterizing  simplicial complexes which have properties like vertex decomposability, shellability and Cohen-Macaulayness are some main problems in combinatorial commutative algebra.  It is rather hopeless to give a full classification of simplicial complexes with each of these properties.
In this regard, finding classes of simplicial complexes, especially independence complexes of graphs with a desired property have been considered by many researchers (cf. \cite{F,FV,HMV,VVi,W,W1}). Constructing new simplicial complexes from the existing ones satisfying a desired property is another way to know more about the characterization. In the works \cite{CN,DE,FH,Moha,V}, the idea of making modifications to a graph like adding whiskers and ears to the graph in order to obtain sequentially
Cohen-Macaulay, Cohen-Macaulay and vertex decomposable graphs is investigated. In \cite{BV}, the authors developed a construction similar to whiskers to build a vertex decomposable simplicial complex $\Delta_{\chi}$ from a coloring $\chi$ of
the vertices of a simplicial complex $\Delta$, and in \cite{BFHV} for colorings of subsets of the vertices, necessary and sufficient conditions are given for this
construction to produce vertex decomposable simplicial complexes.

\medskip
Motivated by the above works and the concept of expansion of a graph in graph theory, in this paper, we
introduce the concept of expansion of simplicial complexes which is a natural generalization of expansion of graphs. Also, we study some properties of this expansion to see how they are related to corresponding properties of the initial simplicial complex. This tool allows us construct
new vertex decomposable and shellable simplicial complexes from vertex decomposable and shellable ones. Moreover, some families of Cohen-Macaulay simplicial complexes are introduced.  We are also interested in knowing how the homological invariants of the Stanley-Reisner ring of a simplicial complex and its expansions are related.

The paper is organized as follows. In the first section, we review some preliminaries from the literature.
In Section 2, first in Theorem \ref{evd} we show that for a simplicial complex $\Delta$, vertex decomposability of $\Delta$ is equivalent to vertex decomposability of an expansion of $\Delta$. Also it is proved that expansions of a shellable  simplicial complex are again shellable (see Theorem \ref{vI}). Moreover, it is shown that under some conditions, expansions of a simplicial complex inherit Cohen-Macaulayness (see Corollaries \ref{cor2}, \ref{cor3}, \ref{cor1} and \ref{CM}). Finally, in Section 3, for a shellable simplicial complex, the projective dimension and the regularity of its Stanley-Reisner ring are compared with the corresponding ones in an expansion of $\Delta$ (see Propositions \ref{pd} and \ref{shreg}).

\section{Preliminaries}

Throughout this paper, we assume that $\Delta$ is a simplicial complex on the vertex set $V(\Delta)=\{x_1, \dots, x_n\}$. The set of facets (maximal faces) of $\Delta$
is denoted by $\mathcal{F}(\Delta)$.

In this section, we recall some preliminaries which are needed in the sequel. We begin with definition of a vertex decomposable simplicial complex. To this aim, we need to recall definitions of the link and the deletion of a face in $\Delta$.
For a simplicial complex $\Delta$ and $F\in \Delta$, the \textbf{link} of $F$ in
$\Delta$ is defined as $$\T{lk}_{\Delta}(F)=\{G\in \Delta: G\cap
F=\emptyset, G\cup F\in \Delta\},$$ and the \textbf{deletion} of $F$ is the
simplicial complex $$\T{del}_{\Delta}(F)=\{G\in \Delta: G\cap
F=\emptyset\}.$$

\begin{defn}\label{1.1}
{\rm A simplicial complex $\Delta$ is  called \textbf{vertex decomposable} if
$\Delta$ is a simplex, or $\Delta$ contains a vertex $x$ such that
\begin{itemize}
\item[(i)] both $\T{del}_{\Delta}(x)$ and $\T{lk}_{\Delta}(x)$ are vertex decomposable, and
\item[(ii)] every facet of $\T{del}_{\Delta}(x)$ is a facet of $\Delta$.
\end{itemize}
A vertex $x$ which satisfies condition (ii) is called a
\textbf{shedding vertex} of $\Delta$.}
\end{defn}

\begin{rem}\label{remark1}
{\rm It is easily seen that $x$ is a shedding vertex of $\Delta$ if and only if no facet of $\T{lk}_{\Delta}(x)$ is a facet of $\T{del}_{\Delta}(x)$.}
\end{rem}

\begin{defn}
{\rm A simplicial complex $\Delta$ is called \textbf{shellable} if there exists an ordering $F_1<\cdots<F_m$ on the
facets of $\Delta$
such that for any $i<j$, there exists a vertex
$v\in F_j\setminus F_i$ and  $\ell<j$ with
$F_j\setminus F_\ell=\{v\}$. We call $F_1,\ldots,F_m$ a \textbf{shelling} for
$\Delta$.}
\end{defn}
The above definition is referred to as  non-pure shellable and
is due to Bj\"{o}rner and Wachs \cite{BW}. In this paper we will
drop the adjective ``non-pure".

\begin{defn}
{\rm A graded $R$-module $M$ is called
\textbf{sequentially Cohen--Macaulay} (over a field $K$) if there exists a
finite filtration of graded $R$-modules $$0=M_0\subset M_1\subset
\cdots \subset M_r=M$$ such that each $M_i/M_{i-1}$ is
Cohen--Macaulay and
$$\dim(M_1/M_0)<\dim(M_2/M_1)<\cdots<\dim(M_r/M_{r-1}).$$}
\end{defn}

For a $\mathbb{Z}$-graded $R$-module $M$, the \textbf{Castelnuovo-Mumford regularity} (or briefly regularity)
of $M$ is defined as
$$\T{reg}(M) = \max\{j-i: \ \beta_{i,j}(M)\neq 0\},$$
and the \textbf{projective dimension} of $M$ is defined as
$$\T{pd}(M) = \max\{i:\ \beta_{i,j}(M)\neq 0 \ \text{for some}\ j\},$$
where $\beta_{i,j}(M)$ is the $(i,j)$th graded Betti number of $M$.

Let $V = \{x_1,\ldots, x_n\}$ be a finite set, and let $\mathcal{E} = \{E_1,\ldots,E_s\}$ be a family of nonempty subsets
of $V$. The pair $\mathcal{H} = (V, \mathcal{E})$ is called a \textbf{simple hypergraph} if for each $i$, $|E_i| \geq 2$  and   whenever
$E_i,E_j\in \mathcal{E}$ and $E_i \subseteq E_j$, then $i =j$. The elements of $V$ are
called the vertices and the elements of $\mathcal{E}$ are called the edges of $\mathcal{H}$. For a hypergraph $\mathcal{H}$, the \textbf{independence complex} of $\mathcal{H}$ is defined as $$\Delta_{\mathcal{H}}=\{F\subseteq V(\mathcal{H}):\ E\nsubseteq F, \text{ for each } E\in \mathcal{E}(\mathcal{H})\}.$$

A simple graph $G=(V(G), E(G))$ is a simple hypergraph with the vertices $V(G)$ and the edges $E(G)$, where each of its edges has cardinality exactly two.
For a simple graph $G$, the \textbf{edge ideal} of $G$ is defined as the ideal $I(G)=(x_ix_j:\ \{x_i,x_j\}\in E(G))$. It is easy to see that $I(G)$ can be viewed as the Stanley-Reisner ideal of the simplicial complex
$\Delta_{G}$ i.e., $I(G)=I_{\Delta_G}$. Also, the \textbf{big height} of $I(G)$, denoted by $\T{bight}(I(G))$, is
defined as the maximum height among the minimal prime divisors of $I(G)$.

A graph $G$ is called vertex decomposable, shellable, sequentially Cohen-Macaulay or Cohen-Macaulay if
the independence complex $\Delta_G$ is vertex decomposable, shellable, sequentially Cohen-Macaulay or Cohen-Macaulay.

A graph $G$ is called \textbf{chordal}, if it contains no induced cycle of length $4$ or greater.

\begin{defn}\label{1.2}
{\rm
A monomial ideal $I$ in the ring $R=K[x_1,\ldots,x_n]$ has \textbf{linear quotients} if there exists an ordering $f_1, \dots, f_m$ on the minimal generators of $I$ such that the colon ideal $(f_1,\ldots,f_{i-1}):_R(f_i)$ is generated by a subset of $\{x_1,\ldots,x_n\}$ for all $2\leq i\leq m$. We show this ordering by $f_1<\dots <f_m$ and we call it \textbf{an order of linear quotients} on $\mathcal{G}(I)$.
Also for any $1\leq i\leq m$, $\set_I(f_i)$ is defined as
$$\set_I(f_i)=\{x_k:\ x_k\in (f_1,\ldots, f_{i-1}) :_R (f_i)\}.$$
We denote $\set_I(f_i)$ by $\set (f_i)$ if there is no ambiguity about the ideal $I$.
}
\end{defn}

A monomial ideal $I$ generated by monomials of degree $d$ has a \textbf{linear resolution} if $\beta _{i,j}(I)=0$ for all $j\neq i+d$. Having linear quotients is a strong tool to determine some classes of ideals with linear resolution. The main tool in this way is the following lemma.

\begin{lem}(See \cite[Lemma 5.2]{F}.)\label{Faridi}
Let $I=(f_1, \dots, f_m)$ be a monomial ideal with linear quotients such that all $f_i$s are of the same degree. Then $I$ has a linear resolution.
\end{lem}

For a squarefree monomial ideal $I=( x_{11}\cdots
x_{1n_1},\ldots,x_{t1}\cdots x_{tn_t})$, the \textbf{Alexander dual ideal} of $I$, denoted by
$I^{\vee}$, is defined as
$$I^{\vee}:=(x_{11},\ldots, x_{1n_1})\cap \cdots \cap (x_{t1},\ldots, x_{tn_t}).$$
For a simplicial complex $\Delta$ with the vertex set $X=\{x_1, \dots, x_n\}$, the \textbf{Alexander dual simplicial complex} associated to $\Delta$ is defined as
$$\Delta^{\vee}=\{X\setminus F:\ F\notin \Delta\}.$$
For a subset $C\subseteq X$, by $x^C$ we mean the monomial $\prod_{x\in C} x$ in the ring $K[x_1, \dots, x_n]$.
One can see that
$(I_{\Delta})^{\vee}=(x^{F^c} \ : \ F\in \mathcal{F}(\Delta)),$
where $I_{\Delta}$ is the Stanley-Reisner ideal associated to $\Delta$ and $F^c=X\setminus F$.
Moreover, one can see that  $(I_{\Delta})^{\vee}=I_{\Delta^{\vee}}$.

The following theorem which was proved in \cite{T}, relates projective dimension and regularity of a
squarefree monomial ideal to its Alexander dual. It is one of our
tools in the study of the projective dimension and regularity of the ring $R/I_{\Delta}$.

\begin{thm}(See \cite[Theorem 2.1]{T}.) \label{1.3}
Let $I$ be a squarefree monomial ideal. Then
$\T{pd}(I^{\vee})=\T{reg}(R/I)$.
\end{thm}

\section{Expansions of a simplicial complex and their algebraic properties}
In this section, expansions of a simplicial complex and their Stanley-Reisner rings are studied. The main goal is to explore
how the combinatorial and algebraic properties of a simplicial complex $\Delta$ and its Stanley-Reisner ring affects on the expansions.
\begin{defn}\label{2.1}
{\rm
Let $\Delta=\langle F_1,\ldots,F_m\rangle$ be a simplicial complex  with the vertex set $V(\Delta)=\{x_1,\ldots,x_n\}$ and $s_1,\ldots,s_n\in \mathbb{N}$ be arbitrary integers.
For any  $F_i=\{x_{i_1},\ldots,x_{i_{k_i}}\}\in \mathcal{F}(\Delta)$, where $1\leq i_1<\cdots<i_{k_i}\leq n$ and any $1\leq r_1\leq s_{i_1},\ldots, 1\leq r_{k_i}\leq s_{i_{k_i}}$, set
$$F_i^{r_1,\ldots, r_{k_i}}=\{x_{i_1r_1},\ldots,x_{i_{k_i}r_{k_i}}\}.$$
 We define the
$(s_1,\ldots,s_n)$-expansion of $\Delta$ to be a simplicial complex with the vertex set $\{\{x_{11},\ldots,x_{1s_1},x_{21},\ldots,x_{2s_2},\ldots,x_{n1},\ldots,x_{ns_n}\}$ and the facets $$\{x_{i_1r_1},\ldots,x_{i_{k_i}r_{k_i}}\} \ :\ \{x_{i_1},\ldots,x_{i_{k_i}}\}\in \mathcal{F}(\Delta), \ (r_1,\ldots,r_{k_i})\in [s_{i_1}]\times \cdots \times [s_{i_{k_i}}]\}.$$
We denote this simplicial complex by $\Delta^{(s_1,\ldots,s_n)}$}.
\end{defn}

\begin{exam}
{\rm Consider the simplicial complex $\Delta=\langle\{x_1,x_2,x_3\},\{x_1,x_2,x_4\},\{x_4,x_5\}\rangle$ depicted in Figure $1$. Then $$\Delta^{(1,2,1,1,2)}=\langle\{x_{11},x_{21},x_{31}\},\{x_{11},x_{22},x_{31}\},\{x_{11},x_{21},x_{41}\},\{x_{11},x_{22},
x_{41}\},\{x_{41},x_{51}\},\{x_{41},x_{52}\}\rangle.$$
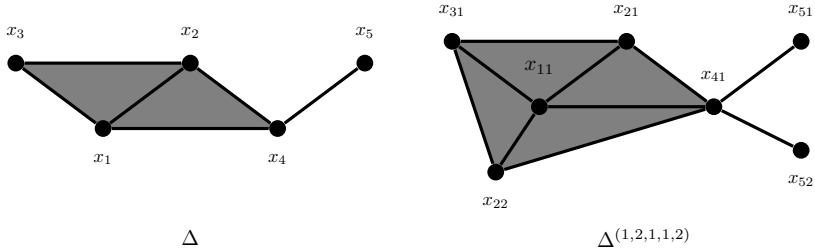
\begin{figure}[h2!]
%\label{fig1}
\begin{center}
\begin{tikzpicture} [scale = .58, very thick = 10mm]

\path[fill,color=gray]
(-24,1.5) -- (-22,3) -- (-20,1.5);
\path[fill,color=gray]
(-24,1.5) -- (-22,3) -- (-26,3);

%\node (n1) at (-10,2)  [Cgray] {1};
 % \node (n6) at (-11.5,0) [Cgray] {$6$};
  \node (n4) at (-20,1.5)  [Cblack] {$14$};
\node () at (-20,.8)  [CW] {$x_4$};
%  \node (n7) at (-10,-2)  [Cgray] {$7$};
%\node (n8) at (-8.5,0)  [Cgray] {$8$};
\node (n3) at (-26,3)  [Cblack] {$13$};
\node () at (-26,3.7)  [CW] {$x_3$};
%\node (n1) at (-15.5,0)  [Cgray] {$1$};
\node (n1) at (-24,1.5)  [Cblack] {$11$};
\node () at (-24,.8)  [CW] {$x_1$};

\node (n2) at (-22,3)  [Cblack] {$12$};
\node () at (-22,3.7)  [CW] {$x_2$};
%\node (n3) at (-12,.5)  [Cgray] {$x_3$};
\node (n5) at (-18,3)  [Cblack] {$15$};
\node () at (-18,3.7)  [CW] {$x_5$};
\node (m) at (-22,-1)  [Cwhite] {$\Delta$};

  \foreach \from/\to in {n1/n4,n1/n2,n2/n4,n4/n5,n1/n3,n2/n3}
    \draw[] (\from) -- (\to);
%%%%%%%%%%%%%%

\path[fill,color=gray]
(-14,2) -- (-12,3.5) -- (-10,2);
\path[fill,color=gray]
(-14,2) -- (-12,3.5) -- (-16,3.5);
\path[fill,color=gray]
(-14,2) -- (-15,.5) -- (-16,3.5);
\path[fill,color=gray]
(-14,2) -- (-15,.5) -- (-10,2);

%\node (n1) at (-10,2)  [Cgray] {1};
 % \node (n6) at (-11.5,0) [Cgray] {$6$};
  \node (n14) at (-10,2)  [Cblack] {$41$};
\node () at (-10,2.7)  [CW] {$x_{41}$};

%\node (n8) at (-8.5,0)  [Cgray] {$8$};
\node (n13) at (-16,3.5)  [Cblack] {$31$};
\node () at (-16,4.2)  [CW] {$x_{31}$};

\node (n22) at (-15,.5)  [Cblack] {$22$};
\node () at (-15,-.2)  [CW] {$x_{22}$};

\node (n11) at (-14,2)  [Cblack] {$11$};
\node () at (-14,2.9)  [Cg] {$x_{11}$};

\node (n12) at (-12,3.5)  [Cblack] {$21$};
\node () at (-12,4.2)  [CW] {$x_{21}$};

\node (n15) at (-8,3.5)  [Cblack] {$51$};
\node () at (-8,4.2)  [CW] {$x_{51}$};

\node (n25) at (-8,1)  [Cblack] {$52$};
\node () at (-8,.3)  [CW] {$x_{52}$};

\node (m) at (-11.6,-1)  [Cwhite] {$\Delta^{(1,2,1,1,2)}$};

  \foreach \from/\to in {n11/n14,n11/n12,n12/n14,n14/n15,n11/n13,n12/n13,n11/n22,n22/n13,n22/n14,n14/n25}
    \draw[] (\from) -- (\to);
%%%%%%%%%%%%%%
\end{tikzpicture}
\label{fig:graph}
\caption{The simplicial complex $\Delta$ and the $(1,2,1,1,2)$-expansion of $\Delta$}\label{Fig1}
\end{center}
\end{figure}
}
\end{exam}

The following definition, gives an analogous concept for the expansion of a hypergraph, which is also a generalization of \cite[Definition 4.2]{FHV}.

\begin{defn}\label{2.3}
{\rm
For a hypergraph $\mathcal{H}$ with the vertex set $V(\mathcal{H})=\{x_1,\ldots,x_n\}$ and the edge set $\mathcal{E}(\mathcal{H})$, we define the
$(s_1,\ldots,s_n)$-expansion of $\mathcal{H}$ to be a hypergraph with the vertex set $\{x_{11},\ldots,x_{1s_1},x_{21},\ldots,x_{2s_2},\ldots,x_{n1},\ldots,x_{ns_n}\}$ and the edge set
\begin{align*}
\{\{x_{i_1r_1},\ldots, x_{i_tr_t}\}:\  \{x_{i_1},\ldots, x_{i_t}\}\in \mathcal{E}(\mathcal{H}),\
(r_1,\ldots,r_{t})\in [s_{i_1}]\times \cdots \times [s_{i_{t}}]\}\cup\\
 \{\{x_{ij},x_{ik}\}: \ 1\leq i\leq n, \ j\neq k\}.
\end{align*}

We denote this hypergraph by $\mathcal{H}^{(s_1,\ldots,s_n)}$.
}
\end{defn}

\begin{rem}
{\rm From Definitions \ref{2.1} and \ref{2.3} one can see that for a hypergraph $\mathcal{H}$ and integers $s_1,\ldots,s_n\in \mathbb{N}$,
$\Delta_{\mathcal{H}^{(s_1,\ldots,s_n)}}=\Delta_{\mathcal{H}}^{(s_1,\ldots,s_n)}.$
Thus the expansion of a simplicial complex is the natural generalization of the concept of expansion in graph theory.
}
\end{rem}

\begin{exam}
{\rm
 Let $G$ be the following graph.
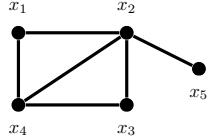
\begin{figure}[h3!]
\label{fig5} \begin{center}
\begin{tikzpicture} [scale = .48, very thick = 10mm]
%\node (n1) at (-10,2)  [Cgray] {1};
 % \node (n6) at (-11.5,0) [Cgray] {$6$};
  \node (n4) at (-23,5)  [Cblack] {$4$};
\node (m7) at (-23,4.3)  [CW] {$x_3$};
%\node (n8) at (-8.5,0)  [Cgray] {$8$};
\node (n3) at (-26,7)  [Cblack] {$3$};
\node (m) at (-26,7.7)  [CW] {$x_1$};
\node (n1) at (-26,5)  [Cblack] {$1$};
\node (m1) at (-26,4.3)  [CW] {$x_4$};
\node (n2) at (-23,7)  [Cblack] {$2$};
\node (m2) at (-23,7.7)  [CW] {$x_2$};
\node (n5) at (-21,6)  [Cblack] {$5$};
\node (m5) at (-21,5.3)  [CW] {$x_5$};
\foreach \from/\to in {n1/n4,n1/n2,n2/n4,n1/n3,n2/n3,n5/n2}
    \draw[] (\from) -- (\to);

%%%%%%%%%%%%%%
\end{tikzpicture}
\label{fig:graph}
%\caption{The simplicial complex $\Delta$ and the $(1,2,1,1,2)$-expansion of $\Delta$}\label{Fig1}
\end{center}
\end{figure}

The graph $G^{(1,1,2,1,2)}$ and the independence complexes $\Delta_G$ and $\Delta_{G^{(1,1,2,1,2)}}$ are
shown in Figure $2$.
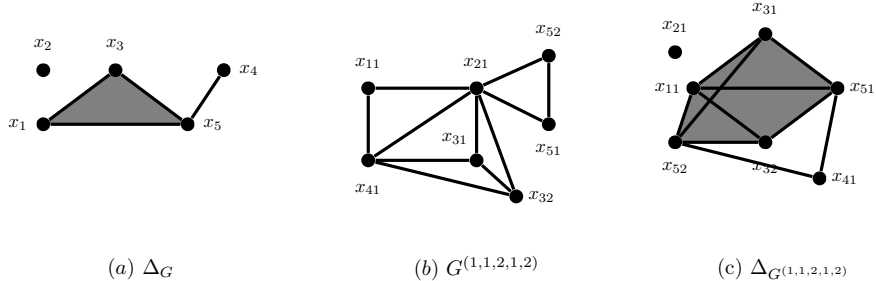
\begin{figure}[h3!]
\label{fig6} \begin{center}
\begin{tikzpicture} [scale = .48, very thick = 10mm]
%\node (n1) at (-10,2)  [Cgray] {1};
 % \node (n6) at (-11.5,0) [Cgray] {$6$};
  \node (n4) at (-23,5)  [Cblack] {$4$};
\node (m7) at (-23.6,5.6)  [CW] {$x_{31}$};

\node (m8) at (-21.2,4)  [CW] {$x_{32}$};
\node (m9) at (-21.9,4)  [Cblack] {$3$};

\node (n3) at (-26,7)  [Cblack] {$3$};
\node (m) at (-26,7.7)  [CW] {$x_{11}$};

\node (n1) at (-26,5)  [Cblack] {$1$};
\node (m1) at (-26,4.2)  [CW] {$x_{41}$};

\node (n2) at (-23,7)  [Cblack] {$2$};
\node (m2) at (-23,7.7)  [CW] {$x_{21}$};

\node (n5) at (-21,6)  [Cblack] {$5$};
\node (m5) at (-21,5.3)  [CW] {$x_{51}$};

\node (n52) at (-21,7.9)  [Cblack] {$5$};
\node (m52) at (-21,8.6)  [CW] {$x_{52}$};
\foreach \from/\to in {n1/n4,n1/n2,n2/n4,n1/n3,n2/n3,n5/n2,n52/n2,n4/m9,m9/n2,m9/n1,n5/n52}
    \draw[] (\from) -- (\to);
\node (m) at (-23,2)  [Cwhite] {$(b)\ G^{(1,1,2,1,2)}$};

%%%%%%%%%%%%%%
\path[fill,color=gray]
(-35,6) -- (-33,7.5) -- (-31,6);

%\node (n1) at (-10,2)  [Cgray] {1};
%  \node (p6) at (-11.5,0) [Cgray] {$6$};
  \node (p5) at (-31,6)  [Cblack] {$5$};
 \node (p7) at (-30.3,6)  [CW] {$x_5$};

 \node (p41) at (-30,7.5)  [Cblack] {$4$};
 \node (p42) at (-29.3,7.5)  [CW] {$x_4$};

\node (p3) at (-35,6)  [Cblack] {$1$};
\node (p10) at (-35.7,6)  [CW] {$x_1$};

\node (p4) at (-33,7.5)  [Cblack] {$3$};
\node (p10) at (-33,8.2)  [CW] {$x_3$};

\node (p2) at (-35,7.5)  [Cblack] {$2$};
\node (p22) at (-35,8.2)  [CW] {$x_2$};
\node (p) at (-32.3,2)  [Cwhite] {$(a)\ \Delta_G$};

  \foreach \from/\to in {p3/p5,p3/p4,p4/p5,p41/p5}
    \draw[] (\from) -- (\to);
%%%%%%%%%%%%%%%%%%
\path[fill,color=gray]
(-17,7) -- (-15,8.5) -- (-13,7);
\path[fill,color=gray]
(-17,7) -- (-13,7) -- (-15,5.5);
\path[fill,color=gray]
(-17.5,5.5) -- (-17,7) -- (-15,8.5);
\path[fill,color=gray]
(-17.5,5.5) -- (-17,7) -- (-15,5.5);

%\node (n1) at (-10,2)  [Cgray] {1};
%  \node (p6) at (-11.5,0) [Cgray] {$6$};
  \node (q5) at (-13,7)  [Cblack] {$5$};
 \node () at (-12.3,7)  [CW] {$x_{51}$};
 %  \node (p7) at (-10,-2)  [Cgray] {$7$};
%\node (p8) at (-8.5,0)  [Cgray] {$8$};
%\node (p2) at (-14,-2)  [Cgray] {$2$};
\node (q1) at (-17.5,5.5)  [Cblack] {$5$};
\node () at (-17.5,4.8)  [CW] {$x_{52}$};

\node (q3) at (-17,7)  [Cblack] {$1$};
\node () at (-17.7,7)  [CW] {$x_{11}$};

\node (q4) at (-15,8.5)  [Cblack] {$3$};
\node () at (-15,9.2)  [CW] {$x_{31}$};

\node (q9) at (-15,5.5)  [Cblack] {$3$};
\node () at (-15,4.8)  [CW] {$x_{32}$};

\node (q14) at (-13.5,4.5)  [Cblack] {$4$};
\node () at (-12.8,4.5)  [CW] {$x_{41}$};

\node (p2) at (-17.5,8)  [Cblack] {$2$};
\node (p22) at (-17.5,8.7)  [CW] {$x_{21}$};
\node (q) at (-14.5,2)  [Cwhite] {(c)\ $\Delta_{G^{(1,1,2,1,2)}}$};

  \foreach \from/\to in {q3/q5,q3/q4,q4/q5,q3/q9,q5/q9,q1/q3,q1/q9,q1/q4,q1/q14,q5/q14}
    \draw[] (\from) -- (\to);

\end{tikzpicture}
\label{fig:graph2}
\caption{The graph $G^{(1,1,2,1,2)}$ and simplicial complexes $\Delta_G$ and $\Delta_{G^{(1,1,2,1,2)}}$} %\label{fig6}
\end{center}
\end{figure}
}
\end{exam}

In the following proposition, it is shown that a graph is chordal if and only if some of its expansions is chordal.
\begin{prop}\label{cl}
For any $s_1,\ldots,s_n\in \mathbb{N}$, $G$ is a chordal graph if and only if $G^{(s_1,\ldots,s_n)}$ is chordal.
\end{prop}

\begin{proof}
If $G^{(s_1,\ldots,s_n)}$ is chordal, then clearly $G$ is also chordal, since it can be considered as an induced subgraph of
$G^{(s_1,\ldots,s_n)}$. Now, let $G$ be chordal, $V(G)=\{x_1,\ldots,x_n\}$ and consider a cycle $C_m: x_{i_1j_1},\ldots, x_{i_mj_m}$ in $G^{(s_1,\ldots,s_n)}$, where $m\geq 4$ and $1\leq j_k\leq s_{i_k}$ for all $1\leq k\leq m$. We consider two cases.

Case 1.  $i_k=i_\ell$ for some distinct integers $k$ and $\ell$ with $1\leq k<\ell\leq m$. Then by the definition of expansion,
$x_{i_kj_k}x_{i_\ell j_\ell}\in E(G^{(s_1,\ldots,s_n)})$. Thus if $x_{i_kj_k}x_{i_\ell j_\ell}$ is not an edge of $C_m$, then it is a chord in $C_m$.
Now, assume that $x_{i_kj_k}x_{i_\ell j_\ell}$ is an edge of $C_m$. Note that since
$x_{i_\ell j_\ell}x_{i_{\ell+1}j_{\ell+1}}\in E(C_m)$, either $i_\ell=i_{\ell+1}$ or $x_{i_\ell}x_{i_{\ell+1}}\in E(G)$ (if $\ell=m$, then set $\ell+1:=1$). Thus $x_{i_kj_k}x_{i_{\ell+1}j_{\ell+1}}\in E(G^{(s_1,\ldots,s_n)})$ is a chord in $C_m$.

Case 2. $i_k\neq i_\ell$ for any distinct integers $1\leq k,\ell\leq m$. By the definition of expansion, one can see that
$x_{i_1},\ldots, x_{i_m}$ forms a cycle of length $m$ in $G$. So it has a chord. Let $x_{i_k}x_{i_\ell}\in E(G)$ be a chord in this cycle. Then
$x_{i_kj_k}x_{i_\ell j_\ell}\in E(G^{(s_1,\ldots,s_n)})$ is a chord in $C_m$.
Thus $G^{(s_1,\ldots,s_n)}$ is also chordal.
\end{proof}

The following theorem illustrates that the vertex decomposability of a simplicial complex is equivalent to the vertex decomposability of its expansions.
\begin{thm}\label{evd}
Assume that $s_1, \dots, s_n$ are positive integers. Then $\Delta$ is vertex decomposable if and only if  $\Delta^{(s_1,\ldots,s_n)}$ is vertex decomposable.
\end{thm}
\begin{proof}
Assume that $\Delta$ is a simplicial complex with the vertex set $V(\Delta)=\{x_1,\dots, x_n\}$ and $s_1, \dots, s_n$ are positive integers. To prove the `only if' part, we use generalized induction on $|V(\Delta^{(s_1,\ldots,s_n)})|$ (note that $|V(\Delta^{(s_1,\ldots,s_n)})|\geq |V(\Delta)|$). If $|V(\Delta^{(s_1,\ldots,s_n)})|=|V(\Delta)|$, then $\Delta=\Delta^{(s_1, \dots, s_n)}$ and so there is nothing to prove in this case. Assume inductively that for all vertex decomposable simplicial complexes $\Delta'$ and all positive integers $s'_1, \dots, s'_n$ with $|V(\Delta'^{(s'_1,\dots,s'_n)})|< t$, $\Delta'^{(s'_1,\ldots,s'_n)}$ is vertex decomposable. Now, we are going to prove the result when $t=|V(\Delta^{(s_1,\ldots,s_n)})|>|V(\Delta)|$. Since $|V(\Delta^{(s_1,\ldots,s_n)})|>|V(\Delta)|$, there exists an integer $1\leq i\leq n$ such that $s_i>1$. If $\Delta=\langle F\rangle$ is a simplex, we claim that $x_{i1}$ is a shedding vertex of $\Delta^{(s_1,\ldots,s_n)}$. It can be easily checked that
$$\T{lk}_{\Delta^{(s_1,\ldots,s_n)}}(x_{i1})=\langle F\setminus \{x_i\}\rangle^{(s_1, \dots, s_{i-1},s_{i+1}, \dots, s_n)}$$
and
$$\T{del}_{\Delta^{(s_1,\ldots,s_n)}}(x_{i1})=\Delta^{(s_1, \dots, s_{i-1}, s_i-1, s_{i+1}, \dots, s_n)}.$$
So, inductive hypothesis ensures that $\T{lk}_{\Delta^{(s_1,\ldots,s_n)}}(x_{i1})$ and $\T{del}_{\Delta^{(s_1,\ldots,s_n)}}(x_{i1})$ are vertex decomposable. Also, it can be seen that every facet of $\T{del}_{\Delta^{(s_1,\ldots,s_n)}}(x_{i1})$ is a facet of $\Delta^{(s_1,\ldots,s_n)}$. This shows that $\Delta^{(s_1,\ldots,s_n)}$ is vertex decomposable in this case. Now, if $\Delta$ is not a simplex, it has a shedding vertex, say $x_1$. We claim that $x_{11}$ is a shedding vertex of $\Delta^{(s_1,\ldots,s_n)}$. To this end, it can be seen that
$$\T{lk}_{\Delta^{(s_1,\ldots,s_n)}}(x_{11})={\T{lk}_\Delta (x_1)}^{(s_2,\ldots,s_n)}$$
and
$$\T{del}_{\Delta^{(s_1,\ldots,s_n)}}(x_{11})=\left\lbrace
\begin{array}{c l}
\Delta^{(s_1-1, s_2,\ldots,s_n)}    & \text{if $s_1>1$;}\\
 {\T{del}_\Delta (x_1)}^{(s_2,\ldots,s_n)} & \text{if $s_1=1$.}
\end{array}
\right.$$
Hence,  inductive hypothesis deduces that $\T{lk}_{\Delta^{(s_1,\ldots,s_n)}}(x_{11})$ and $\T{del}_{\Delta^{(s_1,\ldots,s_n)}}(x_{11})$ are vertex decomposable simplicial complexes. Now, suppose that $F^{j_1, \dots, j_k}=\{x_{i_1j_1}, \dots, x_{i_kj_k} \}$ is a facet of $\T{lk}_{\Delta^{(s_1,\ldots,s_n)}}(x_{11})$, where $F=\{x_{i_1}, \dots, x_{i_k}\}$ is a face of $\Delta$. Then since $\T{lk}_{\Delta^{(s_1,\ldots,s_n)}}(x_{11})={\T{lk}_\Delta (x_1)}^{(s_2,\ldots,s_n)}$,
$F$ is a facet of $\T{lk}_\Delta (x_1)$. So, there is a vertex $x_{i_{k+1}}\in V(\Delta)$ such that $\{x_{i_1}, \dots,x_{i_k}, x_{i_{k+1}}\}$ is a face of $\T{del}_\Delta (x_1)$ (see Remark \ref{remark1}). Hence $\{x_{i_1j_1}, \dots, x_{i_kj_k}, x_{i_{k+1}1} \}$ is a face of $\T{del}_{\Delta^{(s_1,\ldots,s_n)}}(x_{11})$. This completes the proof of the first part.

To prove the `if' part, we also use generalized induction on $|V(\Delta^{(s_1,\ldots,s_n)})|$. If $|V(\Delta^{(s_1,\ldots,s_n)})|=|V(\Delta)|$, then $\Delta=\Delta^{(s_1, \dots, s_n)}$ and so there is nothing to prove in this case. Assume inductively that for all simplicial complexes $\Delta'$ and all positive integers $s'_1, \dots, s'_n$ with $|V(\Delta'^{(s'_1,\dots,s'_n)})|< t$ such that $\Delta'^{(s'_1,\dots,s'_n)}$ is vertex decomposable, we have proved that $\Delta'$ is also vertex decomposable. Now, we are going to prove the result when $t=|V(\Delta^{(s_1,\ldots,s_n)})|>|V(\Delta)|$. Now, since $|V(\Delta^{(s_1,\ldots,s_n)})|>|V(\Delta)|$ and  $\Delta^{(s_1,\ldots,s_n)}$ is vertex decomposable, it has a shedding vertex, say $x_{11}$. If $s_1>1$, then
$$\T{del}_{\Delta^{(s_1,\ldots,s_n)}}(x_{11})=\Delta^{(s_1-1, s_2,\ldots,s_n)},$$
and the inductive hypothesis ensures that $\Delta$ is vertex decomposable as desired. Else, we should have $s_1=1$,
$$\T{lk}_{\Delta^{(s_1,\ldots,s_n)}}(x_{11})={\T{lk}_\Delta (x_1)}^{(s_2,\ldots,s_n)}$$
and
$$\T{del}_{\Delta^{(s_1,\ldots,s_n)}}(x_{11})={\T{del}_\Delta (x_1)}^{(s_2,\ldots,s_n)}.$$
So, inductive hypothesis implies that $\T{lk}_\Delta (x_1)$ and $\T{del}_\Delta (x_1)$ are vertex decomposable simplicial complexes.  Now, assume that $F=\{x_{i_1}, \dots, x_{i_k}\}$ is a facet of $\T{del}_\Delta (x_1)$. Then $\{x_{i_11}, \dots, x_{i_k1}\}$ is a facet of $\T{del}_{\Delta^{(s_1,\ldots,s_n)}}(x_{11})$. Since $x_{11}$ is a shedding vertex of $\Delta^{(s_1,\ldots,s_n)}$, $\{x_{i_11}, \dots, x_{i_k1}\}$ is a facet of $\Delta^{(s_1,\ldots,s_n)}$. Hence, $F$ is a facet of $\Delta$ and the proof is complete.
\end{proof}

\begin{rem}\label{pure}
{\rm By the notations as in Definition \ref{2.1}, $\Delta$ is pure if and only if $\Delta^{(s_1,\ldots,s_n)}$ is pure, since any facet $F_i^{r_1,\ldots, r_{k_i}}$ of $\Delta^{(s_1,\ldots,s_n)}$ has the same cardinality as $F_i$.}
\end{rem}

The following theorem together with Theorem \ref{evd} help us to see how the Cohen-Macaulayness propery in a vertex decomposable simplicial complex and its expansions are related.
\begin{thm}\label{vc}
A vertex decomposable simplicial complex $\Delta$ is Cohen-Macaulay if and only if $\Delta$ is pure.
\end{thm}
\begin{proof}
See \cite[Theorem 11.3]{BW2} and \cite[Theorem 5.3.18]{VIL}.
\end{proof}

\begin{cor}\label{cor2}
Let $\Delta$ be a vertex decomposable simplicial complex and $s_1, \dots, s_n$ be positive integers. Then $\Delta$ is Cohen-Macaulay if and only if $\Delta^{(s_1,\ldots,s_n)}$ is Cohen-Macaulay.
\end{cor}
\begin{proof}
By Theorem \ref{evd}, $\Delta^{(s_1,\ldots,s_n)}$ is also vertex decomposable. Also, by Theorem \ref{vc},  $\Delta$, respectively $\Delta^{(s_1,\ldots,s_n)}$, is Cohen-Macaulay if and only if $\Delta$, respectively $\Delta^{(s_1,\ldots,s_n)}$, is pure. Now, by Remark \ref{pure}, the result is clear.
\end{proof}

\begin{cor}\label{cor3}
Let $G$ be a Cohen-Macaulay chordal graph or a Cohen-Macaulay bipartite graph. Then $G^{(s_1,\ldots,s_n)}$ is Cohen-Macaulay.
\end{cor}
\begin{proof}
By \cite[Corollary 7]{W} and \cite[Corollary 2.12]{VT} chordal graphs and Cohen-Macaulay bipartite graphs are vertex decomposable.  The result now follows from Corollary \ref{cor2}.
\end{proof}

In the following theorem, it is shown that shellability is preserved under expansion and from a shelling for $\Delta$, a shelling for its expansion is constructed.
\begin{thm}\label{vI}
Let $\Delta$ be a shellable  simplicial complex with $n$ vertices. Then  $\Delta^{(s_1,\ldots,s_n)}$ is shellable for any $s_1,\ldots,s_n\in \mathbb{N}$.
\end{thm}
\begin{proof}
Use the notations as in Definition \ref{2.1}. Let $\Delta$ be a shellable simplicial complex with the shelling order $F_1<\cdots<F_m$ on the facets of $\Delta$. Consider an order on $\mathcal{F}(\Delta^{(s_1,\ldots,s_n)})$ as follows.
For two facets $F_i^{r_1,\ldots, r_{k_i}}$ and $F_j^{r'_1,\ldots, r'_{k_j}}$ of $\Delta^{(s_1,\ldots,s_n)}$
\begin{itemize}
\item[(i)] if $i<j$, set $F_i^{r_1,\ldots, r_{k_i}}<F_j^{r'_1,\ldots, r'_{k_j}}$,
\item[(ii)] if $i=j$, set $F_i^{r_1,\ldots, r_{k_i}}<F_i^{r'_1,\ldots, r'_{k_i}}$, when $(r_1,\ldots, r_{k_i})<_{lex} (r'_1,\ldots, r'_{k_i})$.
\end{itemize}
We show that this ordering forms a shelling order. Consider two facets $F_i^{r_1,\ldots, r_{k_i}}$ and $F_j^{r'_1,\ldots, r'_{k_j}}$ with $i<j$.
Since $F_i<F_j$, there exists an integer $\ell<j$ and $x_{j_t}\in F_j\setminus F_i$ such that $F_j\setminus F_\ell=\{x_{j_t}\}$. So
$x_{j_tr'_t}\in F_j^{r'_1,\ldots, r'_{k_j}}\setminus F_i^{r_1,\ldots, r_{k_i}}$. Let
$F_\ell=\{x_{\ell_1},\ldots,x_{\ell_{k_\ell}}\}$, where $\ell_1<\cdots<\ell_{k_\ell}$. Then there exist indices $h_1,\ldots,h_{t-1},h_{t+1},\ldots,h_{k_j}$ such that
$j_1=\ell_{h_1},\ldots, j_{t-1}=\ell_{h_{t-1}},j_{t+1}=\ell_{h_{t+1}},\ldots,j_{k_j}=\ell_{h_{k_j}}$.
Thus
$$F_j^{r'_1,\ldots, r'_{k_j}}\setminus F_\ell^{r''_1,\ldots,r''_{k_\ell}}=\{x_{j_tr'_t}\},$$ where
$r''_{h_1}=r'_1,\ldots,r''_{h_{t-1}}=r'_{t-1},r''_{h_{t+1}}=r'_{t+1},\ldots,r''_{h_{k_j}}=r'_{k_j}$ and $r''_{\lambda}=1$ for other indices $\lambda$. Since $\ell<j$, we have $F_\ell^{r''_1,\ldots,r''_{k_\ell}}<F_j^{r'_1,\ldots, r'_{k_j}}$.

Now assume that $i=j$ and $F_i^{r_1,\ldots, r_{k_i}}<F_i^{r'_1,\ldots, r'_{k_i}}$. Thus $$(r_1,\ldots, r_{k_i})<_{lex} (r'_1,\ldots, r'_{k_i}).$$ Let $1\leq t\leq k_i$ be an integer with $r_t<r'_t$. Then $x_{i_tr'_t}\in F_i^{r'_1,\ldots, r'_{k_i}}\setminus
F_i^{r_1,\ldots, r_{k_i}}$, $$F_i^{r'_1,\ldots,r'_{k_i}}\setminus F_i^{r'_1,\ldots,r'_{t-1},r_t,r'_{t+1},\ldots, r'_{k_i}}=\{x_{i_tr'_t}\}$$
and $$(r'_1,\ldots,r'_{t-1},r_t,r'_{t+1},\ldots, r'_{k_i})<_{lex} (r'_1,\ldots,r'_{k_i}).$$ Thus $F_i^{r'_1,\ldots,r'_{t-1},r_t,r'_{t+1},\ldots, r'_{k_i}}<F_i^{r'_1,\ldots,r'_{k_i}}$. The proof is complete.
\end{proof}

The following corollary is an immediate consequence of Theorem \ref{vI}, Remark \ref{pure} and \cite[Theorem 5.3.18]{VIL}.
\begin{cor}\label{cor1}
Let $\Delta$ be a pure shellable simplicial complex. Then $\Delta^{(s_1,\ldots,s_n)}$ is Cohen-Macaulay for any $s_1,\ldots,s_n\in \mathbb{N}$.
\end{cor}

\begin{thm}\label{one dimension}
Let $\Delta$ be a pure one dimensional simplicial complex. Then the following statements are equivalent.
\begin{itemize}
\item[(i)] $\Delta$ is connected.
\item[(ii)] $\Delta$ is vertex decomposable.
\item[(iii)] $\Delta$ is shellable.
\item[(iv)] $\Delta$ is sequantially Cohen-Macaulay.
\item[(v)] $\Delta$ is Cohen-Macaulay.
\end{itemize}
\end{thm}
\begin{proof}
\begin{itemize}
  \item[$(i\Rightarrow ii)$] Suppose that $\Delta=\langle F_1, \dots, F_m\rangle$. We use induction on $m$. If $m=1$, $\Delta$ is clearly vertex decomposable. Suppose inductively that the result has been proved for smaller values of $m$. We consider two cases. If $\Delta$ has a free vertex (a vertex which belongs to only one facet), then there is a facet, say $F_m=\{x,y\}$, of $\Delta$ such that $x\not\in \bigcup_{i=1}^{m-1}F_i$. In this case
      $\T{lk}_\Delta(x)=\langle \{ y\}\rangle,$
      which is clearly vertex decomposable.
      Also, since $\Delta$ is connected,
      $$\T{del}_\Delta(x)=\langle F_1,\dots, F_{m-1}\rangle$$ is a pure one dimensional connected simplicial complex. So, by inductive hypothesis $\T{del}_\Delta(x)$ is also vertex decomposable. Moreover each facet of $\T{del}_\Delta(x)$ is a facet of $\Delta$. This shows that $\Delta$ is vertex decomposable. Now, suppose that  $\Delta$ doesn't have any free vertex. So, each vertex belongs to at least two facets. Hence, there is a vertex $x$ such that $\T{del}_\Delta(x)$ is also connected and one dimensional. (Note that since $\Delta$ is connected and one dimensional, it may be illustrated as a connected graph. Also, from graph theory, we know that every connected graph has at least two vertices such that by deleting them, we still have a connected graph). Now, by induction hypothesis we have that $\T{del}_\Delta(x)$ is vertex decomposable. Also, $\T{lk}_\Delta(x)$ is a discrete set and so vertex decomposable. Furthermore, in view of the choice of $x$, it is clear that every facet of $\T{del}_\Delta(x)$ is a facet of $\Delta$. Hence, $\Delta$ is vertex decomposable as desired.
  \item[$(ii\Rightarrow iii)$] follows from \cite[Theorem 11.3]{BW2}.
  \item[$(iii\Rightarrow iv)$] is firstly shown by Stanley in \cite{Stanley}.
  \item[$(iv\Rightarrow v)$] The result follows from the fact that every pure sequantially Cohen-Macaulay simplicial complex is Cohen-Macaulay.
  \item[$(v\Rightarrow i)$] follows from \cite[Corollary 5.3.7]{VIL}.
\end{itemize}
\end{proof}

\begin{cor}\label{CM}
Let $\Delta$ be a Cohen-Macaulay simplicial complex of dimension one. Then $\Delta^{(s_1,\ldots,s_n)}$ is Cohen-Macaulay for any $s_1,\ldots,s_n\in \mathbb{N}$.
\end{cor}
\begin{proof}
Since $\Delta$ is Cohen-Macaulay of dimension one, Theorem \ref{one dimension} implies that $\Delta$ is pure shellable. Hence, Corollary \ref{cor1} yields the result.
\end{proof}

The evidence suggests when $\Delta$ is Cohen-Macaulay, its expansions are also Cohen-Macaulay. Corollaries \ref{cor2}, \ref{cor3}, \ref{cor1} and \ref{CM} are some results in this regard. But in general, we did not get to a proof or a counter example for this statement. So, we just state it as a conjecture as follows.

\textbf{Conjecture.} If $\Delta$ is a  Cohen-Macaulay simplicial complex, then $\Delta^{(s_1,\ldots,s_n)}$ is Cohen-Macaulay for any $s_1,\ldots,s_n\in \mathbb{N}$.

\section{Homological invariants of expansions of a simplicial complex}
We begin this section with the next theorem which presents formulas for the projective dimension and depth of the Stanley-Reisner ring of an expansion of a shellable simplicial complex in terms of the corresponding invariants of the Stanley-Reisner ring of the simplicial complex.

\begin{thm}\label{pd}
Let $\Delta$ be a shellable simplicial complex with the vertex set $\{x_1,\ldots,x_n\}$, $s_1,\ldots,s_n\in \mathbb{N}$ and $R=K[x_1,\ldots,x_n]$ and $R'=K[x_{11},\ldots,x_{1s_1},\ldots,x_{n1},\ldots,x_{ns_n}]$ be polynomial rings over a field $K$. Then $$\T{pd}(R'/I_{\Delta^{(s_1,\ldots,s_n)}})=\T{pd}(R/I_{\Delta})+s_1+\cdots+s_n-n$$
and
$$\T{depth}(R'/I_{\Delta^{(s_1,\ldots,s_n)}})=\T{depth}(R/I_{\Delta}).$$
\end{thm}
\begin{proof}

Let $\Delta$ be a shellable simplicial complex. Then it is sequentially Cohen-Macaulay. By Theorem \ref{vI}, $\Delta^{(s_1,\ldots,s_n)}$ is also shellable and then
sequentially Cohen-Macaulay. Thus by \cite[Corollary 3.33]{MVi}, $$\T{pd}(R'/I_{\Delta^{(s_1,\ldots,s_n)}})=\bight(I_{\Delta^{(s_1,\ldots,s_n)}})$$
and $$\T{pd}(R/I_{\Delta})=\bight(I_{\Delta}).$$ Let $k=\min\{|F|:\ F\in \mathcal{F}(\Delta)\}$. It is easy to see that
$\min\{|F|:\ F\in \mathcal{F}(\Delta^{(s_1,\ldots,s_n)})\}=k$. Then $\bight(I_{\Delta})=n-k$ and
$$\bight(I_{\Delta^{(s_1,\ldots,s_n)}})=|V(\Delta^{(s_1,\ldots,s_n)})|-k=s_1+\cdots+s_n-k=s_1+\cdots+s_n+\T{pd}(R/I_{\Delta})-n.$$
The second equality holds by Auslander-Buchsbaum formula, since $\T{depth}(R')=s_1+\cdots+s_n$.
\end{proof}

In the following example, we compute the invariants in Theorem \ref{pd} and illustrate the equalities.
\begin{exam}
{\rm Let $\Delta=\langle \{x_1,x_2,x_3\}, \{x_1,x_2,x_4\},\{x_4,x_5\}\rangle$. Then $\Delta$ is shellable with the order as listed in $\Delta$. Then
$$\Delta^{(1,1,2,1,2)}=\langle \{x_{11},x_{21},x_{31}\},\{x_{11},x_{21},x_{32}\},\{x_{11},x_{21},x_{41}\},\{x_{41},x_{51}\},\{x_{41},x_{52}\}\rangle.$$
computations by Macaulay2 \cite{GS}, show that $\T{pd}(R/I_{\Delta})=3$ and $\T{pd}(R'/I_{\Delta^{(s_1,\ldots,s_n)}})=5=\T{pd}(R/I_{\Delta})+s_1+\cdots+s_n-n=3+1+1+2+1+2-5$.
Also $\T{depth}(R'/I_{\Delta^{(s_1,\ldots,s_n)}})=\T{depth}(R/I_{\Delta})=2$. }
\end{exam}

The following result, which is a special case of  \cite[Corollary 2.7]{Leila}, is our main tool to prove Proposition \ref{shreg}.

\begin{thm}(See \cite[Corollary 2.7]{Leila}.)\label{Leila}
Let $I$ be a monomial ideal with linear quotients with the ordering $f_1<\cdots<f_m$ on the minimal generators of $I$.
Then $$\beta_{i,j}(I)=\sum_{\deg(f_t)=j-i} {|\set_I(f_t)|\choose i}.$$
\end{thm}

\begin{prop}\label{shreg}
Let $\Delta=\langle F_1,\ldots,F_m\rangle$ be a shellable simplicial complex  with the vertex set $\{x_1,\ldots,x_n\}$, $s_1,\ldots,s_n\in \mathbb{N}$ and $R=K[x_1,\ldots,x_n]$ and $R'=K[x_{11},\ldots,x_{1s_1},\\ \ldots,x_{n1},\ldots,x_{ns_n}]$ be polynomial rings over a field $K$. Then
\begin{itemize}
\item[(i)] if $s_1,\ldots,s_n>1$, then
$\T{reg}(R'/I_{\Delta^{(s_1,\ldots,s_n)}})=\T{dim}(\Delta)+1=\T{dim}(R/I_{\Delta});$
 \item[(ii)] if  for each $1\leq i\leq m$, $\lambda_i=|\{x_\ell\in F_i:\ s_\ell>1\}|$, then $$\T{reg}(R'/I_{\Delta^{(s_1,\ldots,s_n)}})\leq \T{reg}(R/I_{\Delta})+\max \{\lambda_i:\ 1\leq i\leq m\}.$$
\end{itemize}
\end{prop}
\begin{proof}
Without loss of generality assume that $F_1<\cdots<F_m$ is a shelling for $\Delta$. We know that $I_{\Delta^{\vee}}$ has linear quotients with the ordering $x^{F_1^c}<\cdots<x^{F_m^c}$ on its minimal generators (see \cite[Theorem 1.4]{HD}). Moreover by Theorem \ref{1.3}, $\T{reg}(R'/I_{\Delta^{(s_1,\ldots,s_n)}})=\T{pd}(I_{\Delta^{(s_1,\ldots,s_n)^{\vee}}})$ and by \cite[Theorem 5.1.4]{BH} we have $\T{dim}(R/I_\Delta)=\T{dim}(\Delta)+1$. Thus, to prove (i), it is enough to show that $\T{pd}(I_{\Delta^{(s_1,\ldots,s_n)^{\vee}}})=\dim(\Delta)+1$. By Theorem \ref{Leila}, $\T{pd}(I_{\Delta^{\vee}})=\max\{|\set(x^{F_i^c})|:\  1\leq i\leq m\}$.
For any $1\leq i\leq m$, $\set(x^{F_i^c})\subseteq F_i$, since any element $x_\ell\in \set(x^{F_i^c})$ belongs to
$(x^{F_j^c}):_R(x^{F_i^c})$ for some $1\leq j<i$. Thus $x_\ell=x^{F_j^c}/\gcd(x^{F_j^c},x^{F_i^c})=x^{F_i\setminus F_j}$.  Let $F_i=\{x_{i_1},\ldots,x_{i_{k_i}}\}$ and $\set(x^{F_i^c})=\{x_{i_\ell}:\ \ell\in L_i\}$, where $L_i\subseteq \{1,\ldots,k_i\}$.
Consider the shelling for $\Delta^{(s_1,\ldots,s_n)}$ constructed in the proof of Theorem \ref{vI}. Using again of \cite[Theorem 1.4]{HD} shows that this shelling induces an order of linear quotients on the minimal generators of $I_{\Delta^{(s_1,\ldots,s_n)^{\vee}}}$. With this order
\begin{equation}\label{tasavi}
\set(x^{(F_i^{r_1,\ldots, r_{k_i}})^c})=\{x_{i_\ell r_\ell}:\ \ell\in L_i\}\cup \{x_{i_tr_t}:\ r_t>1\}.
\end{equation}
More precisely, if $r_t>1$ for some $1\leq t\leq k_i$, then
\begin{align*}
  x_{i_tr_t} & =x^{(F_i^{r_1,\ldots, r_{k_i}}\setminus F_i^{r_1,\ldots,r_{t-1},r_t-1,r_{t+1},\ldots, r_{k_i}})}\\
   & \in (x^{(F_i^{r_1,\ldots,r_{t-1},r_t-1,r_{t+1},\ldots, r_{k_i}})^c}):_{R'}(x^{(F_i^{r_1,\ldots, r_{k_i}})^c}).
\end{align*}
Hence, $x_{i_tr_t}\in \set(x^{(F_i^{r_1,\ldots, r_{k_i}})^c})$.
Also for any $x_{i_\ell}\in \set(x^{F_i^c})$, there exists $1\leq j<i$ such that $x_{i_\ell}=x^{F_i\setminus F_j}\in (x^{F_j^c}):_R(x^{F_i^c})$.
Thus there exist positive integers $r''_1,\ldots,r''_j$ such that
 \begin{align*}
   x_{i_\ell r_\ell} & =x^{(F_i^{r_1,\ldots,r_{k_i}}\setminus F_j^{r''_1,\ldots,r''_{k_j}})}\\
    & \in (x^{(F_j^{r''_1,\ldots, r''_{k_j}})^c}):_{R'}(x^{(F_i^{r_1,\ldots,r_{k_i}})^c}).
 \end{align*}
Hence, $x_{i_\ell r_\ell}\in \set(x^{(F_i^{r_1,\ldots, r_{k_i}})^c})$.
Now, if $s_1,\ldots,s_n>1$, then
$\set(x^{(F_i^{s_{i_1},\ldots, s_{i_{k_i}}})^c})=\{x_{i_1s_{i_1}},\ldots,x_{i_{k_i}s_{i_{k_i}}}\}.$ Thus
$$\T{pd}(I_{{\Delta^{(s_1,\ldots,s_n)}}^\vee})=\max\{|\set(x^{(F_i^{s_{i_1},\ldots, s_{i_{k_i}}})^c})|:\  1\leq i\leq m\}=\max\{|F_i|:\ 1\leq i\leq m \}=\T{dim}(\Delta)+1.$$  To prove (ii), notice that by equality \ref{tasavi}, $|\set(x^{(F_i^{r_1,\ldots, r_{k_i}})^c})|\leq |\set(x^{F_i^c})|+ \lambda_i$. Therefore $\T{pd}(I_{{\Delta^{(s_1,\ldots,s_n)}}^\vee})\leq \T{pd}(I_{\Delta^{\vee}})+\max \{\lambda_i:\ 1\leq i\leq m\}.$
Now, by Theorem \ref{1.3}, the result holds.
\end{proof}

\begin{exam}
{\rm
Consider the chordal graph $G$ depicted in Figure $3$ and its $(2,2,3,2,3)$-expansion which is a graph with $12$ vertices. Then $\Delta_G=\langle\{x_1,x_3\},\{x_3,x_5\},\{x_4,x_5\},\{x_2\}\rangle$.
Since $G$ is shellable, by Proposition \ref{shreg}, $\T{reg}(R'/I(G^{(2,2,3,2,3)}))=\dim(\Delta_G)+1=2$.
\begin{figure}[h4!]
%\label{fig7} %\begin{center}
\begin{tikzpicture} [scale = .48, very thick = 10mm]
%\node (n1) at (-10,2)  [Cgray] {1};
 % \node (n6) at (-11.5,0) [Cgray] {$6$};
  \node (n4) at (-23,5)  [Cblack] {$x_4$};
\node () at (-23,4.3)  [CW] {$x_3$};
%\node (n8) at (-8.5,0)  [Cgray] {$8$};
\node (n3) at (-26,7)  [Cblack] {$3$};
\node () at (-26,7.7)  [CW] {$x_1$};
\node (n1) at (-26,5)  [Cblack] {$1$};
\node () at (-26,4.3)  [CW] {$x_4$};
\node (n2) at (-23,7)  [Cblack] {$2$};
\node () at (-23,7.7)  [CW] {$x_2$};
\node (n5) at (-24.5,9)  [Cblack] {$5$};
\node () at (-24.5,9.7)  [CW] {$x_5$};
\foreach \from/\to in {n1/n4,n1/n2,n2/n4,n1/n3,n2/n3,n5/n2,n3/n5}
    \draw[] (\from) -- (\to);
\node (q) at (-24.5,1)  [Cwhite] { $G$};

%%%%%%%%%%%%%%

\node (n41) at (-12,5)  [Cblack] {$3$};
\node () at (-12,4.3)  [CW] {$x_{31}$};
\node (n42) at (-10.5,3)  [Cblack] {$3$};
\node () at (-10.5,2.3)  [CW] {$x_{32}$};
\node (n43) at (-9,4)  [Cblack] {$3$};
\node () at (-9,3.3)  [CW] {$x_{33}$};

\node (n31) at (-15,7)  [Cblack] {$1$};
\node () at (-14.5,6.5)  [CW] {$x_{11}$};
\node (n32) at (-17,8)  [Cblack] {$1$};
\node () at (-17.5,7.3)  [CW] {$x_{12}$};

\node (n11) at (-15,5)  [Cblack] {$4$};
\node () at (-16,5.1)  [CW] {$x_{41}$};
\node (n12) at (-16,3)  [Cblack] {$4$};
\node () at (-16,2.3)  [CW] {$x_{42}$};

\node (n21) at (-12,7)  [Cblack] {$2$};
\node () at (-13,6.6)  [CW] {$x_{21}$};
\node (n22) at (-9,7.7)  [Cblack] {$2$};
\node () at (-8.3,7.7)  [CW] {$x_{22}$};

\node (n51) at (-13.5,9)  [Cblack] {$5$};
\node () at (-14.1,9.5)  [CW] {$x_{51}$};
\node (n52) at (-11.5,10)  [Cblack] {$5$};
\node () at (-11,10.7)  [CW] {$x_{52}$};
\node (n53) at (-13.5,11)  [Cblack] {$5$};
\node () at (-13,11.7)  [CW] {$x_{53}$};
\foreach \from/\to in {n11/n12,n11/n41,n12/n41,n11/n42,n12/n42,n11/n43,n12/n43,
n21/n22,n11/n22,n11/n21,n11/n22,n12/n21,n12/n22,
n21/n41,n21/n42,n22/n41,n22/n42,n21/n43,n22/n43,
n11/n31,n12/n31,n11/n32,n12/n32,
n21/n31,n21/n32,n32/n22,n22/n31,
n51/n21,n52/n21,n53/n21,n51/n22,n52/n22,n53/n22,
n31/n51,n31/n52,n31/n53,n32/n51,n32/n52,n32/n53,
n31/n32,
n41/n42,n41/n43,n42/n43,n51/n52,n51/n53,n52/n53}
    \draw[] (\from) -- (\to);

\node (q) at (-13.5,1)  [Cwhite] {\ $G^{(2,2,3,2,3)}$};

\end{tikzpicture}
%\label{fig7}
\caption{The graph $G$ and the $(2,2,3,2,3)$-expansion of $G$}\label{fig7}
%\end{center}

\end{figure}
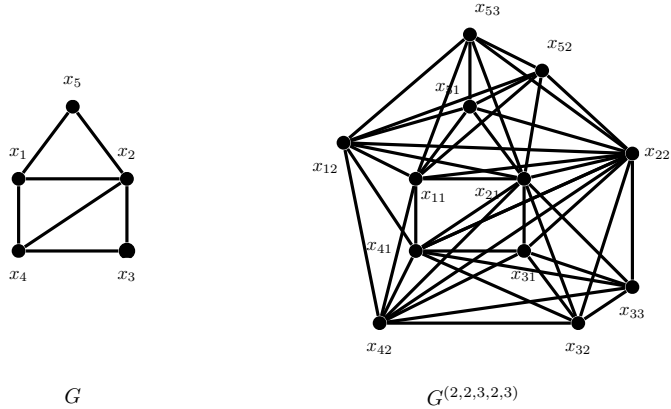
}

\end{exam}

%\textbf{Acknowledgments:}
%The authors would like to thank the referee for careful reading of the paper. The research of the first author was in part supported by a grant %from IPM (No. 93130021).

\providecommand{\bysame}{\leavevmode\hbox
to3em{\hrulefill}\thinspace}

\end{document}